\documentclass[a4paper,11pt]{article}
\usepackage{amsfonts}
\usepackage{graphics}
\usepackage{amsmath, amsthm}  
\usepackage{amssymb,bbm} 

\usepackage{tikz}
\usepackage{epstopdf}

\newcommand{\C}{\mathbb{C}}\newcommand{\id}{\mathbbm{1}}\newcommand{\N}{\mathbb{N}}
\newcommand{\Z}{\mathbb{Z}}

\newcommand{\Lf}{\mathfrak{f}}
\newcommand{\Lg}{\mathfrak{g}}

\newcommand{\Lm}{\mathfrak{m}}

\newcommand{\Lr}{\mathfrak{r}}\newcommand{\Ls}{\mathfrak{s}}

\newcommand{\CA}{\mathcal{A}}

\newcommand{\rank}{\operatorname{rank}}
\newcommand{\fix}{\operatorname{fix}}
\newcommand{\Fix}{\operatorname{Fix}}
\newcommand{\ord}{\operatorname{ord}}
\newcommand{\eig}{\operatorname{eig}}
\newcommand{\dl}{\operatorname{dl}}

\newcommand{\map}[3]{ #1 : #2 \longrightarrow #3 }

\newcommand{\mapl}[5]{ #1 : #2 \longrightarrow #3 : #4 \longmapsto #5 }

\newtheorem{theorem}{Theorem}
\newtheorem{definition}{Definition}
\newtheorem{proposition}{Proposition}
\newtheorem{corollary}{Corollary}

\newtheorem{lemma}{Lemma}

\title{A note on the structure of graded Lie algebras}
\author{Wolfgang Alexander Moens}

\begin{document}

\maketitle

\abstract{Consider a finite-dimensional, complex Lie algebra $\Lg$ and a semisimple automorphism $\alpha$. This note aims to give a short and simple proof for explicit upper bounds for the derived length of the radical $\Lr$ and the rank of a Levi complement $\Lg / \Lr$ in terms of the number of eigenvalues of $\alpha$ and the dimension of the space of fixed-points. This is an extension of classical theorems by Kreknin, Shalev and Jacobson.}

\linespread{1.2}

\section{Introduction}

The literature provides many results on the structure of Lie algebras admitting transformations that are regular in some sense, e.g.: \cite{BorelMostow}. One of the best known results is due to Kreknin and has since been refined in many ways, \cite{Kreknin}. The theorem states that there exists a function $\map{k}{\N}{\N}$ such that every Lie ring with fixed-point free automorphism of finite order $n$ is solvable of derived length at most $k(n)$. More generally, Shalev has shown there is a function $\map{K}{\N}{\N}$ such that every Lie ring that admits a fixed-point free automorphism with $n$ distinct eigenvalues is solvable of derived length at most $K(n)$, \cite{Shalev}. Let us assume $k(n)$ and $K(n)$ assume the minimal values for which the theorems hold. We then have $$k(n) \leq K(n-1) \leq 2^{n-1}.$$ If the automorphism has non-trivial fixed-points however, the Lie ring need not be solvable at all. Indeed: even simple Lie algebras admit automorphisms that are not regular. The literature also provides many results on the structure of Lie algebras that admit automorphisms whose fixed-point subspace is small in some sense. One of the best known results is due to Khukhro and Makarenko, \cite{KhukhroMakarenko}: 

\begin{theorem} If a Lie algebra $\Lg$ admits an automorphism $\alpha$ of finite order $n$ with fixed-point subalgebra of finite dimension $m$, then $\Lg$ has a soluble automorphically-invariant
ideal of derived length bounded above in terms of $n$ and of finite codimension bounded above in
terms of $m$ and $n$. \end{theorem}

In view of these results, we aim to prove the following, (cf. theorem $3$).

\begin{theorem} Let $\Lg$ be a finite-dimensional Lie algebra over the complex numbers. Let $\Lr$ be its solvable radical and let $\alpha$ be a semisimple automorphism of $\Lg$ of with $n$ eigenvalues and $m$-dimensional fixed-point subalgebra. Then the rank of any Levi-complement is bounded by $$\rank(\Lg / \Lr) \leq 2 \cdot m \cdot n$$ and the derived length of the radical is bounded as follows:
 \begin{eqnarray*}
  \dl(\Lr) &\leq& (m + 1) \cdot K(n) + m.
 \end{eqnarray*} 
 \end{theorem}

We end this introduction with some remarks. Both theorem $1$ and $2$ (cf. theorem $3$) reduce to the classical case if the automorphism is regular. Secondly, note that the bound on the derived length by Khukhro and Makarenko is better than that of theorem $2$ in the sense that it does not depend on the dimension of the subspace of fixed points. On the other hand, we would like to point out three advantages of the proofs in this note. Firstly: the proof of theorem $2$ requires much less effort than that of theorem $1$. Secondly, the characteristic ideal is simply the solvable radical of the Lie algebra and its codimenson can therefore be interpreted as the dimension of a Levi complement. Finally, the explicit upper bounds of  corollary $1$ and theorem $3$ below are an easy byproduct of the proofs and smaller than the previous (recursive) upper bounds.

\begin{corollary} Let $\Lg$, $\Lr$, $n$ and $m$ be as above. Then $ \dl(\Lr) \leq (m +1) 2^{n} + m$ and $\dim(\Lg / \Lr) \leq 16  (m n)^3 + 4  (m n)^2 + 224 (m n).$ \end{corollary}

In this note we will only consider finite-dimensional Lie algebras over the complex numbers.

\section{Semi-simple automorphisms of Lie algebras}

Let us introduce some notation. For an automorphism $\alpha$ of $\Lg$ we define $\Fix(\alpha)$ to be the space of fixed-points and we define $\fix(\alpha)$ to be the dimension of that space. Similarly, we let $\eig(\alpha)$ be the number of distinct eigenvalues of $\alpha$. If $\alpha$ is periodic, we let $\ord(\alpha)$ be its order, that is: the minimal $n \in \N$ for which $\alpha^n = \id_{\Lg}$. Note that periodic automorphisms are semisimple.

\subsection{Simple and semi-simple Lie algebras}

The following result quantifies the general observation that automorphisms of simple Lie algebras admit many fixed-points (theorems $6$ and $8$ of \cite{JacobsonFix}) and it is based on the classification of the complex, simple Lie algebras. We would like to point out to the reader that Jacobson's use of the term \emph{regular automorphism} in \cite{JacobsonFix} is not the one used in this paper's introduction (or in much of the literature).

\begin{lemma}[Jacobson] For every simple Lie algebra $\Ls$ and every semisimple automorphism $\alpha$ of $\Ls$, we have the inequality $$\rank(\Ls) \leq 2 \cdot \fix(\alpha).$$ \end{lemma}

Jacobson confined his attention to simple Lie algebras, ``for the sake of simplicity.'' It turns out that the generalization to semisimple Lie algebras is not too difficult. Let us first consider Lie algebras of a special form. Let $\Ls$ be a semisimple Lie algebra that decomposes into $m \in \N_0$ simple ideals $\Ls_0 \oplus \Ls_1 \oplus \cdots \oplus \Ls_{m-1}$ and suppose it has an automorphism $\alpha \in \operatorname{Aut}(\Ls)$ satisfying $\alpha(\Ls_j) = \Ls_{j+1}$. Here, the subscripts are taken modulo $m$ so that $\alpha$ permutes the simple ideals cyclically. In particular: all simple ideals are isomorphic. \newline

Let $\{ \pi_i \}_{0 \leq i \leq \Z_m}$ and $\{ \iota_i \}_{0 \leq i \leq \Z_m}$ be the obvious projections and embeddings. Consider a non-zero eigenvector $v$ with non-zero eigenvalue $\lambda$. Then the homogeneous components $\{ \pi_j(v) \}_{j \in \Z_m}$ are linearly independent and span an $\alpha$-stable subspace $V$ of $\Ls$. Let $\alpha'$ be the restriction of $\alpha$ to $V$. Then $\alpha'$ has $m$ distinct eigenvalues $\{ \omega_m^i \cdot \lambda \}_{0 \leq i < m}$, for some primitive $m$'th root of unity $\omega_m$. In particular: $m \leq \eig(\alpha)$. Let $\alpha_0$ be the automorphism $\pi_0 \circ \alpha^m \circ \iota_0$ of $\Ls_0$. Then $\fix(\alpha) = \fix(\alpha_0)$ and we get
$$ \rank(\Ls) = m \cdot \rank(\Ls_0) \leq 2 \cdot \eig(\alpha) \cdot \fix(\alpha_0) = 2 \cdot \eig(\alpha) \cdot \fix(\alpha) .$$
We have just shown that the proposition below holds if the automorphism permutes the simple ideals of $\Ls$ cyclically. Let us now consider the general case. 

\begin{proposition} For every semisimple Lie algebra $\Ls$ and every semisimple automorphism $\alpha$ of $\Ls$ we have $$ \rank(\Ls) \leq 2 \cdot \eig(\alpha) \cdot \fix(\alpha) .$$ \end{proposition}

\begin{proof} Clearly, $\Ls$ can be decomposed as $\Ls = \Lm_1 \oplus \cdots \oplus \Lm_t$ with each ideal $\Lm_j$ stable under $\alpha$ and such that $\alpha_j =: \pi_j \circ \alpha \circ \iota_j$ permutes the simple ideals of $\Lm_j$ cyclically. Here again, the $\pi_j$ and $\iota_j$ are the obvious projections and embeddings. Since $\rank(\Ls) = \sum_j \rank(\Lm_j)$ and $ \fix(\alpha) = \sum_j \fix(\alpha_j)$, we get
\begin{eqnarray*}
\rank(\Ls) &\leq& \sum_j 2 \cdot \eig(\alpha_j) \cdot \fix(\alpha_j) \\
&\leq& 2 \cdot \eig(\alpha) \cdot \fix(\alpha).
\end{eqnarray*}
\end{proof}

Note that the special case illustrates that, unlike in the simple case, a bound for $\operatorname{rank}(\Ls)$ must depend on both $\eig(\alpha)$ and $\fix(\alpha)$.

\subsection{Solvable Lie algebras}

\begin{proposition} Consider a solvable Lie algebra $\Lg$ with an automorphism $\alpha$ that has exactly $n$ distinct eigenvalues and an $m$-dimensional fixed-point space. Then the derived length of $\Lg$ is at most $(m + 1) \cdot K(n) + m$. \end{proposition}

\begin{proof} Let us proceed by induction on $m = \fix(\alpha)$. Since Shalev's theorem deals with the case $m = 0$, we may assume that $m$ is positive. Let us consider any $\alpha$-stable subnormal series $0 = \Lg_{r+1} \subseteq \Lg_r \subseteq \Lg_{r-1} \subseteq \cdots \subseteq \Lg_1 = \Lg$ of $\Lg$ and a non-zero fixed-point $x$ of $\alpha$. Then there exists an $i$ such that $x \in \Lg_i \setminus \Lg_{i+1}$ and we may assume that $i$ is maximal under this condition. In particular: the induced automorphism on $\Lg / \Lg_i$ is fixed-point free. Let $\Lg'_i$ be an $\alpha$-stable complement to $\C \cdot x$ in $\Lg_i$, containing the commutator $[\Lg_i,\Lg_i]$. In particular: $\Lg_i = \C \cdot x \ltimes \Lg'_i$. Then the $\alpha$-stable subnormal series $$0 = \Lg_{r+1} \subseteq \Lg_r \subseteq \Lg_{r-1} \subseteq \cdots \subseteq \Lg_{i+1} \subseteq \Lg'_i \subseteq \Lg_i \subseteq \cdots \subseteq \Lg_1 = \Lg$$ has one more term than the previous one, and $\fix(\alpha|_{\Lg'_i}) = \fix(\alpha) - 1$. We may now apply the induction hypothesis to $\Lg'_i$ to obtain $\dl(\Lg'_i) \leq m \cdot K(n) + (m-1)$. We finish the proof by combining the above:
\begin{eqnarray*}
\dl(\Lg) &\leq& \dl(\Lg / \Lg_i) + \dl(\Lg_i / \Lg'_i) + \dl(\Lg'_i) \\
& \leq & K(n) + 1 + m \cdot K(n) + (m-1) \\
&=& (m + 1 ) \cdot K(n) + m.
\end{eqnarray*}
\end{proof}

\begin{proposition} Consider a solvable Lie algebra $\Lg$ with an automorphism $\alpha$ of order $n$ and an $m$-dimensional fixed-point space. Then the derived length of $\Lg$ is at most $(m + 1) \cdot k(n) + m$. \end{proposition}

\begin{proof} The above proof also works in this case. Since $\alpha$ is periodic, it is semisimple. The base case, $m=0$, corresponds with Kreknin's theorem. It then suffices to perform the substitution $K(n) \mapsto k(n)$. \end{proof}

\subsection{Reduction to solvable and semisimple Lie algebras}

The proof of theorem $2$ is now straightforward. Let $\Lg$ be a finite-dimensional Lie algebra over the complex numbers. Let $\Lr$ be the solvable radical of $\Lg$ and $\Ls$ a Levi-complement to $\Lr$ in $\Lg$ so that we may write $\Lg = \Ls \ltimes \Lr$. Let $\alpha$ be a semisimple automorphism of $\Lg$. Since $\Lr$ is known to be a characteristic ideal of $\Lg$, $\Lr$ is stable under $\alpha$ and the restriction $\map{\alpha|_{\Lr}}{\Lr}{\Lr}$ to $\Lr$ defines an automorphism of $\Lr$. We may also define an automorphism $\map{\overline{\alpha}}{\Ls}{\Ls}$ on the quotient $\Ls \cong \Lg / \Lr$ in the obvious way. (Alternatively: a result by Mostow yields a Levi-complement that is stable under $\alpha$ and we may denote the restriction of $\alpha$ to that complement by $\overline{\alpha}$.) Clearly $\eig(\alpha|_{\Lr}), \eig(\overline{\alpha}) \leq \eig(\alpha)$ and $\fix(\alpha|_{\Lr}),\fix(\overline{\alpha}) \leq \fix(\alpha)$. If $\alpha$ is periodic of order $n \in \N$, then also $\alpha|_{\Lr}$ and $\overline{\alpha}$ are periodic and their orders divide $n$. \newline

We may therefore restrict our attention to semisimple automorphisms of solvable and semisimple Lie algebras. Proposition $2$ then yields the upper bound for the derived length. Proposition $1$ gives the upper bound for the rank of the Levi complement. \newline

The classification of the finite-dimensional, simple, complex Lie algebras implies the following bound.

\begin{lemma} Consider the monotone function $\mapl{f}{\N}{\N}{k}{2 k^2 + 2k + 112}$. Then $\dim(\Ls_0) \leq f(\operatorname{rank}(\Ls_0))$ for each simple $\Ls_0$. \end{lemma}

\begin{proof} (of Corollary $1$) In order to obtain the upper bound for the derived length, we simply apply the upper bound for the function $K(n)$ as given by Shalev, for example: $K(n) \leq 2^{n-1}$. \newline

Now let $\Ls$ be semisimple and let $\Ls_0$ be a simple ideal of maximal dimension. Then the number of simple ideals is bounded from above by $\operatorname{rank}(\Ls)$ so that $\dim(\Ls) \leq \operatorname{rank}(\Ls) \cdot f(\operatorname{rank}(\Ls_0))$. Since $f$ is monotone, we obtain $\dim(\Ls) \leq \operatorname{rank}(\Ls) \cdot f(\operatorname{rank}(\Ls))$. \end{proof}

By making the obvious substitutions, we obtain the following.

\begin{theorem} Let $\Lg$ be a finite-dimensional Lie algebra over the complex numbers. Let $\Lr$ be its solvable radical and let $\alpha$ be an automorphism of $\Lg$ of order $n$ and with an $m$-dimensional fixed-point subalgebra. Then $$\dim(\Lg / \Lr) \leq 16 (m n)^3 + 4 (m n)^2 + 224 (m n) $$ and $
 \dl(\Lr) \leq (m + 1) \cdot k(n) + m \leq (m+1) \cdot 2^{n-1} + m.$
 \end{theorem}

\section{Remarks}

There is a classical theorem by Higman that refines Kreknin's theorem. It states that Lie algebras with a regular automorphism of prime order $p$ are nilpotent of $p$-bounded class, \cite{Higman} (cf. \cite{BorelSerre} and \cite{JacobsonRegular}). This was later generalized by Khukhro and Makarenko to: every Lie algebra with an automorphism of order $p \in \mathbb{P}$ and $m$-dimensional space of fixed-points contains an automorphically invariant ideal of $(p,m)$-bounded codimension and $p$-bounded class, \cite{KhukhroMakarenko}. One might therefore expect the following analogue of proposition $2$ and theorem $3$ to be true:

\begin{quote} \textbf{Conjecture.} \emph{The nilradical of a Lie algebra that admits an automorphism of prime order $p$ with $m$-dimensional space of fixed-points, is nilpotent of $(p,m)$-bounded class.} \end{quote}

The standard filiform Lie algebras form counter-examples to this conjecture since each such algebra admits an involutive automorphism with one-dimensional fixed-point space. Finally, we remark that Higman \cite{Higman}, Khukhro \cite{KhukhroBookComput}, Hughes \cite{Hughes}, Shumyatsky and Tamarozzi \cite{ShumyatskyTamarozzi} have confirmed the conjecture $k(n) \leq n-1$ for $n$ up to $7$, so that we obtain the following.

\begin{corollary} Let $\Lr$ be the solvable radical of a Lie algebra $\Lg$ with periodic automorphism $\alpha$ of order at most $7$. Then $$\rank(\Lg / \Lr) \leq 2 \cdot \ord(\alpha) \cdot \fix(\alpha)$$ and $$\dl(\Lr) \leq \ord(\alpha) \cdot \fix(\alpha) + \ord(\alpha) + \fix(\alpha) .$$ \end{corollary}

\section*{Aknowledgements}

The author would like to thank his host Efim Zelmanov, N. Wallach, L. Small, and E. I. Khukhro for their insightful comments.


\begin{thebibliography}{99999}

\bibitem{BorelMostow} {\sc Borel, A.; Mostow, G. D.}: \emph{On semi-simple automorphisms of Lie algebras.} Ann. of Math. (2) 61, (1955). 389Ð405. 

\bibitem{BorelSerre} {\sc Borel, A.; Serre, J.-P.}: \emph{Sur certains sous-groupes des groupes de Lie compacts.} Comment. Math. Helv. 27, (1953). 128Ð139. 

\bibitem{Higman} {\sc Higman, G.}: \emph{Groups and rings having automorphisms without non-trivial fixed elements.} J. London Math. Soc. 32 (1957), 321Ð334. 

\bibitem{Hughes} {\sc Hughes, I.}: \emph{Groups with fixed-point-free automorphisms.} C. R. Math. Rep. Acad. Sci. Canada 7 (1985), no. 1, 61Ð66. 

\bibitem{JacobsonRegular} {\sc Jacobson, N.}: \emph{A note on automorphisms and derivations of Lie algebras.} Proc. Amer. Math. Soc. 6, (1955). 281Ð283.

\bibitem{JacobsonFix} {\sc Jacobson, N.}: \emph{A note on automorphisms of Lie algebras.} Pacific J. Math. 12 1962 303Ð315.

\bibitem{KhukhroBookComput} {\sc Khukhro, E. I.}: \emph{Nilpotent groups and their automorphisms.}
de Gruyter Expositions in Mathematics, 8. Walter de Gruyter \& Co., Berlin, 1993. xiv+252 pp. ISBN: 3-11-013672-4

\bibitem{KhukhroBook} {\sc Khukhro, E. I.}: \emph{p-automorphisms of finite p-groups.} Mathematical Society Lecture Note Series, 246. Cambridge University Press, Cambridge, 1998. xviii+204 

\bibitem{KhukhroMakarenko} {\sc Khukhro, E. I., Makarenko, N. Yu.}: \emph{Automorphically-invariant ideals satisfying multilinear identities, and group-theoretic applications.} J. Algebra 320 (2008), no. 4, 1723Ð1740. 

\bibitem{Kreknin} {\sc Kreknin, V. A.}: \emph{Solvability of a Lie algebra with a regular automorphism.} Sibirsk. Mat. \v Z. 8 1967 715Ð716. 

\bibitem{Shalev} {\sc Shalev, Aner}: \emph{Automorphisms of finite groups of bounded rank.} Israel J. Math. 82 (1993), no. 1-3, 395Ð404. 

\bibitem{ShumyatskyTamarozzi} {\sc Shumyatsky, P.; Tamarozzi, A.}: \emph{On $\Z_6$-graded Lie rings.} J. Algebra 277 (2004), no. 2, 703Ð716. 


\end{thebibliography}
\end{document}